\documentclass{proc-l}
\usepackage{hyperref}
\usepackage{xcolor}
\usepackage{amsmath}
\usepackage[T1]{fontenc}
\usepackage{listings}
\usepackage{amssymb}
\usepackage{enumitem}

\RequirePackage{tikz-cd}

\newcommand{\R}{\mathbf{R}}
\newcommand{\C}{\mathbf{C}}
\newcommand{\Z}{\mathbf{Z}}
\newcommand{\K}{\mathbf{K}}
\newcommand{\Prems}{\mathbb{P}}
\newcommand{\Proj}{\mathbf{P}}
\newcommand{\N}{\mathbf{N}}
\newcommand{\Q}{\mathbf{Q}}

\newcommand{\F}{\mathbf{F}}

\newcommand{\G}{\mathbf{G}}
\newcommand{\T}{\mathbf{T}}

\newcommand{\base}{\mathbf{b}}
\newcommand{\cone}{\text{C}_{\text{eff}}(X)}

\DeclareMathOperator{\Spec}{Spec}

\DeclareMathOperator{\Pic}{Pic}

\DeclareMathOperator{\eff}{eff}

\DeclareMathOperator{\id}{Id}

\DeclareMathOperator{\Div}{Div}
\DeclareMathOperator{\proj}{proj}
\DeclareMathOperator{\grp}{grp}
\DeclareMathOperator{\Hom}{Hom}

\DeclareMathOperator{\Ext}{Ext}

\DeclareMathOperator{\Vol}{Vol}

\DeclareMathOperator{\covol}{covol}
\DeclareMathOperator{\Br}{Br}

\DeclareMathOperator{\NS}{NS}

\newcommand{\stackT}{\mathcal{T}}

\newcommand{\affine}{\mathbf{A}}
\newcommand{\sheafL}{\mathcal{L}}

\newcommand{\sheafO}{\mathcal{O}}

\newcommand{\w}{\omega}

\newtheorem{theorem}{Theorem}[section]
\newtheorem{cor}[theorem]{Corollary}
\newtheorem{conjecture}[theorem]{Conjecture}
\newtheorem{lemma}[theorem]{Lemma}
\newtheorem{prop}[theorem]{Proposition}

\theoremstyle{definition}
\newtheorem{definition}[theorem]{Definition}
\newtheorem{convention}[theorem]{Convention}
\newtheorem{example}[theorem]{Example}

\theoremstyle{remark}
\newtheorem{remark}[theorem]{Remark}

\newtheorem*{Notations}{Notations}

\numberwithin{equation}{section}

\begin{document}

\title[Multi-height analysis on toric varieties]{Multi-height analysis of rational points of toric varieties}


%    author one information
\author{Bongiorno Nicolas}
\address{}
\curraddr{}
\email{nicolas.bongiorno@univ-grenoble-alpes.fr}
\thanks{}

%    \subjclass is required.

\date{}

\dedicatory{}

%    "Communicated by" -- provide editor's name; required.
%\commby{}

%    Abstract is required.
\begin{abstract}
We study the multi-height distribution of rational points of smooth, projective and split toric varieties over $\Q$ using the lift of the number of points to universal torsors.
\end{abstract}

%Then, using toric stacks, we show that the same method can be applied to the multi-height study of $\Z$-Campana points of $(X,D)$ where $D$ is a certain type of $\Q$-effective divisors over $X$.

\maketitle

\section{Introduction}

For a quasi-Fano variety $V$ with infinitely many rational points, one may asymptotically study the finite set of rational points with a bounded height $H$ associated with the anticanonical line bundle $\w_{V}^{-1}$. In \cite{FrankeManinTschinkel1989}, \cite{Batyrev1990} and \cite{peyre_duke}, Batyrev, Franke, Manin, Tschinkel and Peyre gave strong evidence supporting conjectures relating the asymptotic behavior of the number of points of bounded height on open subsets of $V$ to geometric invariants of $V$.

In the case of toric varieties, the asymptotic behavior was studied in \cite{batyrev_anisotrop_tori} and \cite{batyrev_toric_varieties} using harmonic analysis on adelic groups. In \cite{salberger_torsor}, Salberger was the first to explicitly use universal torsors to study points of bounded height. In particular, he was able to give a new proof of the theorem of Batyrev and Tschinkel for smooth projective split toric varieties over $\Q$. His work highlighted how the constant in the conjectural asymptotic formula is related to the product of local densities on universal torsors. This was later studied in a more general setting by Peyre in \cite{peyre_zeta_height_function} and \cite{methode_du_cercle_peyre}.

More recently in \cite{Peyre_beyond_height}, Peyre proposed a new framework to study the asymptotic distribution of rational points with bounded height. Instead of considering a height relative to one line bundle, it is natural to consider all possible heights. The purpose of this paper is to study the multi-height distribution on rational points of smooth, proper and split toric varieties over $\Q$. We prove that the expected asymptotic behavior holds for these varieties. In particular, we show that there is no accumulating subset for the multi-height distribution in the toric case. This is because we consider the rational points whose multi-height is in the interior of the dual of the effective cone.

A result concerning this all heights approach was given in \cite{demirhan2022distribution}, using height zeta functions. However, the condition $\beta = \pi^{\vee}(u)$ where $u \in (C_{\eff}(X)^{\vee})^{\circ}$ seems to be missing in \cite[theorem 1.1]{demirhan2022distribution}. If one considers the case of the projective line $\Proj^1$ with $(\beta_1,\beta_2) = (1,2)$, their result would imply that the asymptotic behavior is of the form $C.B^3$ whereas it is known to be of the form $C.B^2$.

In any case, our method using descent is totally different. It is inspired by the work of Salberger in \cite{salberger_torsor} and uses the method of counting the number of points of a lattice in a bounded domain developed in \cite{davenport_geometry_numbers} by Davenport.

\textbf{Structure. } We begin in Section 2 by recalling the multi-height formalism and the conjecture on the asymptotic distribution of rational points for \textit{quasi-Fano} varieties associated with this formalism. We then state the main result established in this article for toric varieties. In the following section, we introduce a crucial concept for proving our result: the universal torsor. We explain how to lift the height of a point locally to the universal torsor, describe the specific case of toric varieties, and conclude by linking the measure on universal torsors with the Tamagawa number of the variety. In the final section, we establish our result using estimation techniques from the geometry of numbers and the Möbius inversion formula.

\textbf{Acknowledgment. }I would like to thank my PhD advisor Emmanuel Peyre for all the remarks and suggestions he made during the writing process of this article. I would also like to thank the referee for their help in improving the exposition.

\section{Terminology and the precise statement}

\subsection{Arakelov heights}

In this paragraph, we recall,  following Peyre \cite{Peyre_beyond_height}, how to define a system of Arakelov heights for a smooth, proper and geometrically integral variety $V$ over a number field $K$.

\begin{Notations}
    We denote by $M_K$ the set of places of $K$, by $M_K^0$ the set of finite places and by $M_K^{\infty}$ its infinite places.

    If $v \in M_K^0$, we write $q_v = \sharp \F_v$ the cardinal of the residual field associated to $v$, while if $v \in M_K^{\infty}$, we set $q_v = \exp([K_v:\R])$.

    To $w \in M_K$ which divides $v \in M_{\Q}$, we define $|-|_w $ over $K_w$ by setting for any $x \in K_w$ :
    $$ | x |_w = | N_{K_w/\Q_v}(x) |_v .$$It is an absolute value except when $K_w$ is $\C$ in which case it is the square of an absolute value.

    To an invertible sheaf $\sheafL$ over a scheme $T$, we associate the line bundle $L$ whose sections correspond to $\sheafL$. 
    %$W$ is a covariant functor which gives an equivalence of categories between the inversible sheaves over $T$ and the line bundle over $T$.
    We denote $L^{\times} = L \setminus z(T)$, where $z(T)$ is the closed subscheme given by the zero section. %Again, we have an equivalence of categories given by the functor $L \mapsto L^{\times}$ between the line bundle and the $\G_m$-torsors over $T$.
\end{Notations}

We begin by defining the notion of an adelic norm over a vector bundle.

\begin{definition}
    Let $E \xrightarrow{e} V$ be a vector bundle over $V$.

   For any field extension $L$ over $K$, for any $P \in V(L)$, we denote by $E_P \subset E(L)$ the $L$-vector space which corresponds to the fiber $e^{-1}(P)$ of $e$ over $P$.
   
   An adelic norm $E$ over $e$ is the data of a family of continuous applications $(||-|| )_{v \in M_K}$ 
   $$ ||-||_v : E(K_v) \rightarrow \R_+$$
   such that:
\begin{enumerate}
    \item If $v \in M_K^0$, for any $P \in V(K_v)$, the restriction of $||-||_v$ to $E_P$ defines an ultrametric norm with values in the image of the absolute value $|-|_v$ i.e. $q_v^{\Z}$.

    \item If $v$ is a real place, for any $P \in V(K_v)$, the restriction to $E_P$ is a Euclidean norm.

    \item If $v$ is a complex place, for any $P \in V(K_v)$, the restriction to $E_P$ is the square of a Hermitian norm.

    \item There exists a finite set of places $S \subset M_K$ (which contains $M_K^{\infty}$) and a model $\mathcal{E} \rightarrow \mathcal{V}$ of $E \rightarrow V$ over $\sheafO_S$ such that for any place $v \in M_K \setminus S$, for any $P \in \mathcal{V}(\sheafO_v)$,
    $$ \mathcal{E}_P = \{ y \in E_P | \ \ ||y||_v \leqslant 1 \} .$$
    
\end{enumerate}

\end{definition}

\begin{remark}
    For smooth, proper and split toric varieties over $\Q$ such models always exists over $\Z$.
\end{remark}

The following example is particularly important, because, as we shall later see, a natural system of heights over toric varieties can be deduced from this example.

\begin{example}
    Let $N \in \N^*$. For $v \in M_K$, we consider the map:
    \begin{align*}
        ||-||_v &:  \ K_v^{N+1} \rightarrow \R_+ \\
        &(y_0,..,y_N) \longmapsto \max\limits_{i = 0}^N |y_i|_v
    \end{align*}
    This defines an adelic norm $(||-||_v)_{v \in M_K}$ on the line bundle $\sheafO_{\Proj^N}(-1) \rightarrow \Proj^N$ which corresponds to the $\G_m$-torseur $\affine^{N + 1} \setminus \{ 0 \} \rightarrow \Proj^N$.
\end{example}

\begin{definition}
    If $L \rightarrow V$ is a line bundle, $\{ ||-||_v \}$ an adelic norm on $L$, then we define the height $H_L : V(K) \rightarrow \R_+$ for any $x \in V(K)$ by:
    $$H_L(x) = \prod\limits_{v \in M_K} ||y||_v^{-1}$$where $y \in L^{\times}(K)$ lies above $x$.
\end{definition}

\begin{example}\label{norme_adelique_fubini}
    In particular, by the previous example, for any $x \in K^{N+1}-\{0\}$, if we denote by $P = [x_0:..:x_n] \in \Proj^N(K)$ its image, we get: 
    $$ H_{\sheafO_{\Proj^N}(1)}(P) =  \prod\limits_{v \in M_K}  \max\limits_{0 \leq i \leq N} |x_i|_v . $$
\end{example}

Now, we recall the definition of the group of Arakelov heights $\mathcal{H}(V)$ over $V$:

\begin{definition}
    Let $L$ and $L'$ be two adelic line bundles over $V$ (i.e. equipped with an adelic norm). Let $(||-||_v)$ be the adelic norm on $L$.

    We say that $L$ and $L'$ are equivalent as adelic line bundles when there exists $M > 0$ and a family $(\lambda_v)_v$ of nonnegative real numbers such that:
    \vspace{0,1 cm}
    \begin{enumerate}
        \item $\{v | \lambda_v \neq 1\}$ is finite and $\prod\limits_v \lambda_v = 1$
        \item there exists an isomorphism of adelic line bundles from $L^{\otimes M}$ to $L'^{\otimes M} $, where the adelic norm on $L^M$ is induced by the adelic norm on $L$ given by $\{ (\lambda_v . ||-||_v)_{v \in M_K} \} $.
    \end{enumerate}
    We denote by $\mathcal{H}(V)$ the set of equivalence classes of adelic line bundles. This set is naturally equipped with a structure of group given by the tensorial product of adelic line bundles.
\end{definition}

Hence, we can define the notion of system of heights and of multi-height over \nolinebreak $V$.

\begin{definition}
    A \textbf{system of heights} over $V$ corresponds to the data of a section $s : \Pic(V) \rightarrow \mathcal{H}(V)$ of the forgetful morphism $$ \mathrm{o} : \mathcal{H}(V) \rightarrow \Pic(V) .$$

    Such system of heights defines a \textbf{multi-height map}:
    $$ h : V(K) \rightarrow \Pic(V)^{\vee}_{\R}$$
    such that for any $P \in V(K)$, $h(P)$ is the linear form defined for any $L \in \Pic(V)$ by:
    $$ h(P)(L) = \log H_{s(L)}(P) .$$
\end{definition}

\begin{remark}
    The multi-height is well-defined as $h(P)(L)$ does not depend on the choice of the representative of the equivalence class (see \cite[remark 4.2]{Peyre_beyond_height}).
\end{remark}

\subsection{The multi-height conjecture}

In order to state the conjecture made by Peyre \cite[question 4.8]{Peyre_beyond_height}, we shall define a measure $\nu$ on $\Pic(V)^{\vee}_{\R}$ as follows:

\begin{definition}\label{mesure_picard_group}
    $$ \nu(\mathrm{D}) := \int\limits_{\mathrm{D}} e^{\langle \w_V^{-1},y \rangle} dy $$for any compact subset $\mathrm{D}$ of $\Pic(V)^{\vee}_{\R}$, where the Haar measure $dy$ on $\Pic(V)^{\vee}_{\R}$ is normalized so that the covolume of the dual lattice of the Picard group is one.
\end{definition}
 
For any domain $\mathrm{D} \subset \Pic(V)^{\vee}_{\R}$, we write:
$$ V(K)_{h \in \mathrm{D}} = \{ P \in V(K) \mid h(P) \in \mathrm{D} \} .$$

\begin{definition}
    We shall say that $V$ is quasi-Fano if it satisfies the following hypotheses:
    \vspace{0,1 cm}
    \begin{enumerate}
        \item A multiple of the class of $\w_{V}^{-1}$ is the sum of an ample divisor and a divisor with normal crossings;
        \item The set $V(\Q)$ is Zariski dense;
        \item The group $H^i(V,\sheafO_V) = \{ 0 \}$ for $i =1, 2$;
        \item The geometric Brauer group $\Br(\overline{V})$ is trivial and the geometric Picard group $\Pic(\overline{V})$ has no torsion;
        \item The closed cone $\text{C}_{\eff}(\overline{V})$ is generated by the classes of a finite set of effective divisors.
    \end{enumerate}
\end{definition}

Any Fano variety and any smooth, proper and split toric variety is quasi-Fano. Let us quickly explain where the last claim comes from. The effective cone of such a variety is generated by the divisors associated to each edge of the fan defining the toric variety, and its anticanonical bundle is the sum of these divisors, so it is in the interior of the effective cone. Hence we get that the first and the fifth hypotheses are verified. As such a variety is covered by open subschemes isomorphic to $\affine^d$, where $d$ is the dimension of the toric variety, we get the second assertion and the fact that its geometrical Brauer group is trivial. For such toric variety, by \cite[proposition 4.2.5]{CoxLittleSchenck2011}, the geometric Picard group is torsion free. Finally, by \cite[theorem 9.2.3]{CoxLittleSchenck2011}, the third hypothesis is verified. Peyre conjectures the following in \cite[question 4.8]{Peyre_beyond_height}:

\begin{conjecture}\label{manin_peyre_conjecture}
We assume that our variety is quasi-Fano and that the effective cone $C_{\eff}(V)$ is generated by a finite number of line bundles. Let $\mathrm{D}_1$ be a compact polyhedron of $\Pic(V)^{\vee}_{\R}$ and u be an element of the interior of the dual of the effective cone $(C_{\eff}(V)^{\vee})^{\circ}$. For a real number $B > 1$, we set:
$$ \mathrm{D}_B = \mathrm{D}_1 + \log(B) u .$$
 
There exists an exceptional subset $T \subset V(K)$ such that we have an asymptotic behaviour of the form:
$$ \sharp ( V(K)-T )_{h \in \mathrm{D}_B} \underset{B \rightarrow +\infty}{\sim} \nu(\mathrm{D}_1 ) \beta(V) \tau(V) B^{\langle \w_V^{-1},u \rangle}$$ where $\beta(V) = \sharp H^1(\Q,\Pic(V_{\overline{\Q}}))$ and $\tau(V)$ is the Tamagawa number of $V$ (see definition \ref{tamaga_number_variety}).

\end{conjecture}

\subsection{The Manin-Peyre conjecture in the case of toric varieties}

Let $X$ be a smooth, proper and split toric variety over $\Q$. We shall use in the rest of this paper that the anticanonical line bundle $\w_X^{-1}$ is big, i.e. is in the interior of the effective cone $\cone$. 

\begin{Notations}\label{notations_varietes_toriques}
We denote by $T \subset X$ its open torus, by $\Sigma$ a fan associated to $X = X_{\Sigma}$, by $\Sigma(1)$ the set of edges of $\Sigma$ and by $m$ the number of edges. We denote by $t$ the rank of the Picard group $\Pic(X)$. With our hypothesis, it is a free abelian group of finite rank. To an edge $\rho \in \Sigma(1)$, we associate the $T$-invariant divisor $D_{\rho}$.

We write $$\pi : \Div_T(X) = \Z^{\Sigma(1)} \rightarrow \Pic(X)$$ the natural map which associates to a $T$-invariant divisor $D$ the class of the associated line bundle in $\Pic(X)$ which we write $[D] = \pi(D)$. Recall that $$\w_{X}^{-1} = \sum\limits_{\rho \in \Sigma(1)} [D_{\rho}] .$$We denote by $N = X_*(T)$ the 1-parameter subgroups of $T$ and by $M = X^*(T)$ its dual, the group of characters of $T$.

\end{Notations}

\begin{convention}\label{convention_adelic_norm_toric_variety}
    A line bundle $L$ over $X$ is endowed with the associated natural adelic norm $(||-||_v)_{v \in M_K}$ following \cite[definition 9.2]{salberger_torsor} (see also \cite[proposition 2.1.2]{Batyrev1990}).
\end{convention}

Let us explain how one can interpret this adelic norm in the case of $\Proj^N$.

\begin{example}
    For $0 \leq i \leq N$, we denote by $D_{i}$ the divisor $\G_m^{N}$-invariant of $\Proj^N$ corresponding to the hyperplane $\{ X_i = 0 \} \subset \Proj^N$. Then the adelic norm over $\sheafO_{\Proj^N}(1)$ associated to $D_i$ corresponds to the one of the example \ref{norme_adelique_fubini}.
\end{example}

\begin{proof}
    This is a consequence of \cite[proposition 9.8]{salberger_torsor}.
\end{proof}

\begin{remark}
We can moreover observe that the construction of the adelic norm in \cite[definition 9.2]{salberger_torsor} over $\sheafO(D)$ depends only on the class of the $T$-invariant divisor $D$ modulo $M = X^*(T)$, that is to say its class $[D]$ in $\Pic(X)$.
\end{remark}We shall later see using universal torsor that this way of defining a height system over our toric variety $X$ (and a representative of its system of heights) makes the count of rational points of our variety easier.

\begin{Notations}
    Let $l$ be the smallest positive integer such that there exists $l$ rays of $\Sigma$ not contained in a cone of $\Sigma$. Note that $l \geqslant 2$.
\end{Notations}

The main theorem of this article states that the conjecture \ref{manin_peyre_conjecture} for the multi-height asymptotic behaviour of rational points of toric varieties is true:

\begin{theorem}\label{theorem_principal_multi_hauteur}
Let $\mathrm{D}_1$ be a finite union of compact polyhedrons in $\Pic(X)^{\vee}_{\R}$ and $u$ be an element of the interior of the dual of the effective cone $(C_{\eff}(X)^{\vee})^{\circ}$. For a real number $B > 1$, we set:
$$ \mathrm{D}_B = \mathrm{D}_1 + \log(B) u .$$Let $\varepsilon \in ]0,1-\frac{1}{l}[$. Then we have an asymptotic behaviour of the form:

$$ \sharp ( X(\Q) )_{h \in \mathrm{D}_B} \underset{B \rightarrow +\infty}{=} \nu(\mathrm{D}_1 ) . \tau(X) .  B^{\langle \w_X^{-1},u \rangle} \left( 1 + O\left(B^{-(1-\frac{1}{l} - \varepsilon).\min\limits_{\rho \in \Sigma(1)} \langle [D_{\rho}] , u \rangle} \right) \right) $$ where $\nu$ is the measure over $\Pic(X)^{\vee}_{\R}$ defined as in \ref{mesure_picard_group}.

\end{theorem}

\begin{remark}
In forthcoming work, we will see that, provided that we restrict ourselves to a generating family of 
$\Pic(X)$ consisting of effective line bundles, it is possible to obtain a similar result without imposing a lower bound on the height, but this introduces a $\log(B)$ factor in the error term. We hope also to show that this result implies the usual form of Manin's conjecture.
\end{remark}

\begin{example}\label{exemple_compact_polyhèdre}
    If $ \mathbf{b} = (L_1,..,L_r)$ is a family of line bundles such that their classes $\R$-generate the Picard group $\Pic(X)_{\R}$, $\alpha_1,..,\alpha_r,\beta_1,..,\beta_r$  are real numbers such that for any $i$, $\alpha_i < \beta_i$, then we define a compact polyhedron of $\Pic(X)^{\vee}_{\R}$: $$\mathrm{D}_1 = \{ a \in \Pic(X)^{\vee}_{\R} \mid \alpha_i \leqslant a(L_i) \leqslant\beta_i \} .$$All compact polyhedrons of $\Pic(X)^{\vee}_{\R}$ are of this form.
\end{example}

\section{The notion of universal torsors}

In this section, we shall present how the study of the multi-height distribution of rational points of toric varieties naturally lifts to a counting problem in the universal torsor.

\subsection{Generalities about universal torsors}

Originally, Colliot-Thélène and Sansuc defined universal torsors using a spectral sequence (see \cite[section 2]{ColliotTheleneSansuc1987}) that can be generalized over an noetherian base scheme $S$. For this paragraph, we essentially refer to \cite[section 5]{salberger_torsor}.

We suppose that $ f : V \rightarrow S$ is a smooth proper surjective morphism of constant relative dimension with geometrically connected fibers such that $R^2 f_* \sheafO_V = 0 $ and the relative Picard functor $\Pic_{V/S}$ is étale locally the constant sheaf associated with an ordinary free abelian group of finite rank. We also suppose that $f$ has a rational section. We shall denote by $T_{\NS}$ the group scheme associated to $\Pic_{V/S}$ by duality. Then we have the following exact sequence, functorial in $S$ and $V$ (see \cite[section 5, p.163]{salberger_torsor} which generalizes \cite[2.0.2]{ColliotTheleneSansuc1987}):

$$ 0 \rightarrow H^1(S,T_{\NS}) \xrightarrow{i_1} H^1(V,T_{\NS}) \xrightarrow{\chi} \Hom_{\grp}(\Pic_{V/S},\Pic_{V/S}) \xrightarrow{\partial} H^2(S,T_{\NS})$$

\begin{definition}
    A $T_{\NS}$-torsor $\stackT \rightarrow V$ is said to be universal if its class $$[\stackT] \in H^1(V,T_{\NS})$$ is sent to $\id_{\Pic_{V/S}}$ via $\chi$.
\end{definition}

In the case we will be interested, $S = \Spec(\Z)$, and our scheme will be split that is to say  $\Pic_{V/S}$ is the constant sheaf associated with a free abelian group of finite rank. With this hypothesis, we have that for any $S$-scheme $T$, $\Pic_{V/S}(T) = H^1(V \underset{S}{\times} T, \G_m) $ and $\Hom_{\grp}(\Pic_{V/S},\Pic_{V/S}) = \Hom_{\grp}(\Pic(V),\Pic(V))$. Moreover, the group morphism $\chi$ can be defined for any $[\stackT] \in H^1(X,T_{\NS}) $ as follows (see \cite[Definition 2.3.2 and Theorem 2.3.6]{Skorobogatov2001}):
\begin{align*}
    \chi([\stackT]) : \ &H^1(V,\G_m) \rightarrow H^1(V,\G_m) \\
    &[L^{\times}] \longmapsto [\stackT \overset{T_{\NS}}{\times_{\chi_L}} \G_m ]
\end{align*}where $\chi_L$ is the character of $T_{\NS}$ associated to $[L^{\times}] \in H^1(V,\G_m)$ i.e. $[L] \in \Pic(V)$.

Let us choose $\base = (L_1,..,L_t)$ a basis of $\Pic(V)$ and write $$\stackT_{\base} = L_1^{\times} \times_V ... \times_V L_t^{\times} $$where an element $t \in T_{\NS}$ acts on $L_i^{\times}$ by multiplication by $[L_i](t) \in \G_m$. The torus $T_{\NS}$ is isomorphic to $\G_m^t$ via the characters given by the basis $\base$. Hence we get a natural left action of $T_{\NS}$ over $$q: \stackT_{\base} \rightarrow V$$ which makes $\stackT_{\base}$ a $T_{ \NS}$-torsor over $V$. By \cite[proposition 8.1]{peyre_counting_points_torsors}, we get that the previous construction is independent of the choice of the basis of $\Pic(V)$ because it is unique up to $T_{ \NS}$-equivariant isomorphism. To simplify the notations, we shall denote it $\stackT \rightarrow V$. If each $L_i$ has a non constant global section, we can show that $\chi([\stackT]) = - \id_{\Pic_{V/S}}$.

\begin{remark}
    By functoriality of this cohomological characterization, if $\stackT \rightarrow V$ is a universal torsor and $S' \rightarrow S$ is a base change, then $\stackT_{S'} \rightarrow V_{S'}$ is again a universal torsor.
\end{remark}

Let us recall the following proposition from \cite[proposition 5.15]{salberger_torsor}, which is valid under our assumptions that $\Pic_{V/S}$ is a constant sheaf and $V$ has a rational point:

\begin{prop}
    There is a universal torsor over $V$. Moreover the isomorphism classes of universal torsors over $V$ are parametrized by the $H^1(S,T_{\NS})$-orbits in $H^1(V,T_{\NS})$ defined by $\chi^{-1}(\id_{\Pic_{V/S}})$.
\end{prop}

Cohomological arguments give that if $\alpha \in H^1(V,T_{\NS})$ is the class of a universal torsor $\stackT \rightarrow V$, $x \in V(S)$ and $P \rightarrow S$ is a $T_{\NS}$-torsor associated to the class $\alpha(x)$ then $\stackT \overset{T_{\NS}}{\times} P \rightarrow V$ is a universal torsor such that the fiber over $x$ has a rational point. This come from the fact that if $\beta \in H^1(V,T_{\NS}) $ is the class of $\stackT \overset{T_{\NS}}{\times} P \rightarrow V$, we have by construction $\beta(x) = 0$.

Suppose again that $S = \Spec(\Z)$ and that $\Pic_{V/\Z}$ is the constant sheaf associated to a free abelian group of finite rank $t$. As $\Z$ is a principal ring, we have that $H^1(\Z,T_{\NS} ) = \{1\}$. Thus there is up to isomorphism a unique universal torsor $q : \stackT \rightarrow V$. Moreover, we have:

$$ q(\stackT(\Z)) = V(\Z) .$$

%In this subsection, let us suppose again that $V$ is a smooth, projective and geometrically integral variety over a number field $K$. We add the hypothesis that $\Pic(\Bar{V})$ is free of finite rank  $t = \rank(\Pic(\Bar{V}))$.

%$H^0(\Bar{V},\G_m) = \Bar{K}^{\times} $.

%\begin{Notations}
    %We denote by $T_{N S}$ the Neron-Severi torus whose character group is given by the $\Gal(\Bar{K}/K)$-lattice $\Pic(\Bar{V})$. In the case where $V$ is split over $K$, $T_{N S}$ naturally identifies to $\G_m^t$ choosing a basis of $\Pic(V)$.
%\end{Notations}

%\begin{definition}\label{definition_peyre_universal_torsor}
    %A $T_{N S}$-torsor $q : \stackT \rightarrow V$ would be called universal when, after the base change from $K$ to $\Bar{K}$, it is isomorphic as $\G_m^t$-torsor to $ p  : \stackT_{\base} \rightarrow \Bar{V}$. where $\base$ is any basis of $\Pic(\Bar{V})$.
%\end{definition}

%\begin{question}
   %Dans le théorème précédent, EST CE QUE L'HYPOTHESE LES $L_i$ sont finiment engendrés suffit ?
%\end{question}

\subsection{Local heights defined over the universal torsors}

We keep the hypotheses of the previous paragraph on $V$. By \cite[remark 4.19]{Peyre_beyond_height}, the universal torsor over $V_{\Q}$ is characterized by a universal property as a \textit{pointed torsor}.

\begin{definition}
    A pointed variety over $\K$ is a variety $V$ equipped with a chosen rational point $x \in V(\K)$. A pointed torsor over $V$ is a torsor $P \xrightarrow{ \pi } V$ with a rational point $t \in P(\K)$ such that $\pi(t) = x$.
\end{definition}

We write $q : \stackT \rightarrow V$ for the universal torsor we have chosen. We fix $y_0 \in \stackT(\Q)$ and we denote by $x_0 \in V(\Q)$ its image via $q$. We can now state the following proposition:

\begin{prop}
    For any pointed torsor $P \rightarrow V$ which sends $p_0 \in P(\Q)$ to $x_0 \in V(\Q)$, under a multiplicative group scheme $T$, there exists a unique group morphism $\varphi : T_{N S} \rightarrow T$ and a unique $\varphi$ -équivariant morphism $\Phi_P : \stackT \rightarrow P$ over $V$ which is compatible with the marked point. That is to say the following diagram is commutative:
    \begin{center}
\begin{tikzcd}
{(\stackT,y_0)} \arrow[rr, "\exists! \Phi_P"] \arrow[rd] &         & {(P,p_0)} \arrow[ld] \\
                                                       & {(V,x_0)} &                   
\end{tikzcd}
    \end{center}
\end{prop}

We recall the construction of Peyre \cite[construction 4.27]{Peyre_beyond_height} to define the local height associated to the universal torsor. Suppose we have a given system of heights $$s : \Pic(V) \rightarrow \mathcal{H}(V) .$$We fix a place $v_0 \in M_{\Q}$. For any line bundle $L$, we have by the previous proposition a unique morphism $\Phi_L : \stackT \rightarrow L^{\times}$ which is equivariant for the group morphism $\chi_L : T_{N S} \rightarrow \G_m$ given by $[L] \in X^*(T_{NS}) = \Pic(V)$. This morphism is unique up to multiplication by a scalar $\lambda \in \Q^{\times}$.

For each $[L]$ in $\Pic(V)$, we choose a representative $(||-||_v)$ of the adelic norm associated to $s([L])$. Then we define for any $v \in M_{\Q}$ the map:

\begin{align*}
    ||-||_{L,v} : \stackT&(\Q_v) \rightarrow \R \\
    &y \longmapsto \frac{||\phi_L(y)||_v}{||\phi_L(y_0)||_v} \text{  si } v \neq v_0 \\
    &y \longmapsto \frac{||\phi_L(y)||_{v_0}}{||\phi_L(y_0)||_{v_0}} H_L(q(y_0))^{-1}
\end{align*}

\begin{definition}\label{definition_local_height}
    The local height in the place $v \in M_{\Q}$ associated to the universal torsor $\stackT$ is defined on $\stackT(\Q_v)$ by:
        $$H_{\stackT,v}(y,L) := ||y||_{L,v}^{-1} .$$
    We also define a local logarithm height: $$h_{\stackT,v} : \stackT(\Q_v) \rightarrow \Pic(V)^{\vee}_{\R} $$such that for any $L \in \Pic(V)$, we have:
    $$ H_{\stackT,v}(y,L) = q_v^{\langle h_{\stackT,v}(y) , L \rangle } .$$
\end{definition}

In particular, we have for any $L \in \Pic(V)$ and for any $y \in \stackT(\Q)$,
\begin{equation}\label{local_height_global_height}
     H_L(q(y)) = \prod\limits_{v \in M_{\Q}} H_{\stackT,v}(y,L) .
\end{equation}
In order to describe one of the key property of this local logarithm height, we need to recall the definition of the local logarithm map associated to a torus.

 \begin{definition}
    Let $T$ be a split torus over $\Q$, we denote $M = X^*(T)$ the lattice corresponding to the character group of $T$ and by $N = X_*(T)$ its dual for the usual pairing. For $v \in M_{\Q}$, we have a natural map:
    $$\log_{T,v} : T({\Q}_v) \rightarrow N_{\R}$$such that for any $\chi \in M$, for any $t \in T({\Q}_v)$, we have: $$ \langle \log_{T,v}(t) , \chi \rangle := - \log_{q_v} \left| \chi(t) \right|_v .$$
 \end{definition}

With our convention, we can rewrite the following lemma (\cite[lemma 4.30]{Peyre_beyond_height}):

\begin{lemma}\label{hauteur_locale_torseur_univ_et_action_de_tns}
    For any $t \in T_{\NS}({\Q}_v)$, for any $y \in \stackT({\Q}_v)$,
    $$h_{\stackT,v}(t.y) = \log_{T_{\NS},v}(t) + h_{\stackT,v}(y) .$$
\end{lemma}

\subsection{Universal torsors for split, smooth and projective toric varieties}\label{subsection_univ_torsor_toric_varieties}

We keep the notations from \ref{notations_varietes_toriques}. Let $X$ be a toric variety over $\Q$ (smooth, proper and split) and denote by $\Sigma$ an associated fan. In our particular case, the universal torsor $\stackT_{\Sigma} \rightarrow X$ can be described as an open subset of $\affine^{\Sigma(1)}$ obtained by cutting out a certain closed subset $Z(\Sigma)$.

As in \cite[chapter 5, §1]{CoxLittleSchenck2011}, we define for a cone of the fan $\sigma \in \Sigma$, the monomial:
$$ X^{\Tilde{\sigma}} = \prod\limits_{\rho \not\in \sigma(1)} X_{\rho} .$$Then we shall use the following definition of the exceptional set $Z(\Sigma)$:

\begin{definition}
    We denote by $B(\Sigma)$ the ideal of $\Z[X_{\rho},\rho \in \Sigma(1)]$ generated by the monomials $X^{\Tilde{\sigma}} = \prod\limits_{\rho \not\in \sigma(1)} X_{\rho}$ where $\sigma$ belongs to $\Sigma_{\max}$ the set of cones of the fan $\Sigma$ of maximal dimension. Then we define the closed subset of $\affine^{\Sigma(1)}$ as:
    \vspace{0,1 cm}
    $$Z(\Sigma) =  V(B(\Sigma)) = \bigcap\limits_{\sigma \in \Sigma_{\max}} \left( \bigcup\limits_{\rho \not\in \sigma(1)} \left( X_{\rho} = 0 \right) \right).$$We define $\stackT_{\Sigma} = \affine^{\Sigma(1)} - Z(\Sigma)$.
\end{definition}In particular, this means that for $y \in \Z^{\Sigma(1)}- \{0\}$, we have the following equivalence:
$$y \in \stackT_{\Sigma}(\Z) \Leftrightarrow \gcd(y^{\Tilde{\sigma}}, \sigma \in \Sigma_{\max}) = 1 .$$

\begin{remark}
    There is a more intuitive description of the universal torsor $\stackT_{\Sigma} \rightarrow X$  as an open subset of $\affine^{\Sigma(1)}$ obtained by cutting out a certain closed subset which can be described in term of primitive collections which are defined as follows (see \cite[definition 5.1.5]{CoxLittleSchenck2011}):

\begin{definition}\label{def_primitive_collection}
    A subset $C \subset \Sigma(1)$ is a primitive collection if:
\begin{itemize}
    \item $C \not\subset \sigma(1)$ for any $\sigma \in \Sigma$
    \item for every proper subset $C' \subsetneq C$, there exists $\sigma \in \Sigma$ such that $C' \subset \sigma(1)$
\end{itemize}
\end{definition}

Hence $Z(\Sigma)$ can be described as a union of irreducible components as in \cite[proposition 5.1.6]{CoxLittleSchenck2011}:
\begin{prop}\label{description_torseur_univ_primitive_collection}
    \vspace{0,1 cm}
    $$Z(\Sigma) = \bigcup\limits_{C} V( x_{\rho} \mid \rho \in C ) $$where $C$ belongs to the set of primitive collections.
\end{prop}
\end{remark}

Let $\beta : \Div_T(X)^{\vee} = \Z^{\Sigma(1)} \rightarrow N$ be the $\Z$-linear map which sends the vector of the canonical basis $e_{\rho}$ to $u_{\rho}$ the minimal generator of the edge $\rho$. Denote by $\Sigma_0$ the fan of $\stackT_{\Sigma}$ (see \cite[proposition 5.19]{CoxLittleSchenck2011}). Both toric varieties $\stackT_{\Sigma}$ and $X$ are defined over $\Z$ via their fans. We state here the following essential proposition:

\begin{prop}\label{universal_torsor_over_toric_variety}
    The natural quotient morphism $q : \stackT_{\Sigma} \rightarrow X$ over $\Z$ where $T_{\NS}$ acts on the coordinate associated to $\rho \in \Sigma(1)$ by multiplication by $[D_{\rho}](t)^ {-1}$ makes $\stackT_{\Sigma}$ the universal torsor (unique up to isomorphism) over $X$.
\end{prop}

\begin{proof}
    See \cite[Theorem 5.9]{bongiorno_toric_stacks}.
\end{proof}

We shall denote by $T_{\Sigma(1)}$ the open torus of the toric variety $\stackT_{\Sigma}$, which is naturally isomorphic to $\G_m^{\Sigma(1)}$ via the character 
    \begin{align*}
        \chi_{D_{\rho}} : \ &T_{\Sigma(1)} \rightarrow \G_m \\
        &y \longmapsto y_{\rho} .
    \end{align*}Note that we have $q^{-1}(T) = T_{\Sigma(1)}$.
Recall we have the following exact sequence associated to the toric variety $X$:

$$ 0 \rightarrow M \xrightarrow{\beta^*} \Div_T(X) \xrightarrow{\pi} \Pic(X) \rightarrow 0 .$$To an effective divisor $D = \sum\limits_{\rho \in \Sigma(1)} a_{\rho} D_{\rho}$, we associate the global section $X^D$ in the Cox Ring $\Q[X_{\rho} | \rho \in \Sigma(1) ]$ of the toric variety $X_{\Sigma}$ defined by: $$X^D = \prod\limits_{\rho \in \Sigma(1)} X_ {\rho}^{a_{\rho}} .$$

%voir aussi la citation de Demazure dans Salberger pour généraliser aux champs toriques

\begin{lemma}\label{application_of_demazure_kleiman_criterion}
    There exists a basis $L_1,..,L_t$ of $\Pic(X)$ such that each $L_i$ is a very ample line bundle.
\end{lemma}

\begin{proof}
    Using Kleiman's criterion (see \cite[theorem 1.27]{debarre2001higher}), we can choose $(L_1,..,L_t)$ a basis of $\Pic(X)$ where each $L_i$ is an ample line bundle. Moreover by \cite[theorem 6.1.16]{CoxLittleSchenck2011}, because $X$ is smooth, $L_i$ ample implies $L_i$ very ample. We have proven the lemma.
\end{proof}

We choose such basis. For $1 \leqslant k \leqslant t$, we set $h^0(X,L_k) = r_k$ and we take $\mathcal{M}_k  = \pi^{-1}(L_k) \cap \N^{\Sigma(1)} \subset \Div_T(X)$. Thanks to the description of the Cox Ring of the toric variety $X$, it is a family of $T$-invariant divisor such that:

\begin{equation}\label{base_section_globale}
H^0(X,L_k) = \bigoplus\limits_{D \in \mathcal{M}_k} \Q. X^D .
\end{equation}

We have the following lemma:

\begin{lemma}\label{catesian_diagram_very_ample_basis}
    The closed immersion given by the very ample line bundles:
    $$X_{\Q} \rightarrow \prod\limits_{i = 1}^t \Proj^{r_i-1}_{\Q}    $$is in fact defined over $\Z$ and we have the following commutative diagram which is cartesian:
    \begin{center}
\begin{tikzcd}
\stackT_{\base} \arrow[r] \arrow[d] & \prod\limits_{i = 1}^t \affine^{r_i} \setminus \{ 0 \} \arrow[d] \\
X \arrow[r]                     & \prod\limits_{i = 1}^t \Proj^{r_i-1}          
\end{tikzcd}
    \end{center}where $\base = (L_1,..,L_t)$ and the horizontal maps are toric morphism.   
\end{lemma}

\begin{proof}
    Let us denote by $A_k \in M_{r_k,m}(\Z) $ the matrix whose lines are given by the coordinates of $D \in \mathcal{M}_k$ in the basis $\{ D_{\rho} | \rho \in \Sigma(1) \}$. The $\Z$-linear map given by the matrix $\Bigg[\begin{smallmatrix}
&A_1\\
&\vdots \\
&A_t
\end{smallmatrix}\Bigg]$, that is to say:
   $$ \Phi : \Z^{\Sigma(1)}\rightarrow \bigoplus\limits_{1 \leqslant i \leqslant t} \Z^{r_i}$$
  induces the toric morphism of toric varieties over $\Z$:
   $$ \mathcal{T}_{\Sigma} \rightarrow \prod\limits_{i = 1}^t \affine^{r_i} \setminus \{ 0 \}$$
   which corresponds to the cartesian diagram \ref{catesian_diagram_very_ample_basis}. Moreover, the induced morphism by quotient $$ X \rightarrow \prod\limits_{i = 1}^t \Proj^{r_i-1}$$is again a toric morphism of toric varieties.
\end{proof}

\subsubsection{An alternative description of the system of heights over toric varieties}

Recall that we have equipped the toric variety $X$ with its natural system of heights as in convention \ref{convention_adelic_norm_toric_variety}. We fix the point $1 = (1)_{\rho \in \Sigma(1)}$ of $\stackT(\Q)$ and $q(1) = 1 \in X(\Q)  $. In this paragraph, we shall explain how to lift the height of a rational point $P \in X(\Q)$ to the universal torsor, that is to say, we shall prove the following theorem:

\begin{theorem}\label{simplification_of_the_height_to_univ_torsor}
    Let $y \in \stackT(\Z)$ and $P = q(y) \in X(\Q)$. We have:
    $$ h(P) = h_{\stackT,\infty}(y) .$$
\end{theorem}

\begin{prop}\label{fonctorialite_systeme_hauteur}
   The system of heights $s_X : \Pic(X) \rightarrow \mathcal{H}(X)$ used in \cite{salberger_torsor} on a toric variety $X$ is functorial for the morphisms of split toric varieties.

    That is to say, for $X \xrightarrow{f} Y$ a toric morphism of split toric varieties (i.e. compatible with the actions of the tori), the following diagram is commutative:
    \begin{center}
\begin{tikzcd}
\mathcal{H}(Y) \arrow[r, "f^*"]     & \mathcal{H}(X)    \\
\Pic(Y) \arrow[r, "f^*"'] \arrow[u, "s_Y"] & \Pic(X) \arrow[u, "s_X"]
\end{tikzcd}
\end{center}

\end{prop}

\begin{proof}
For a split toric variety $X$, we denote by $T_{\NS,X}$ its Néron–Severi torus and by $\stackT_X$ its universal torsor, as defined previously. 
According to \cite[section 3]{geraschenko_satriano_stacky_fan}, a toric morphism 
$$X \xrightarrow{f} Y$$ 
between toric varieties lifts to a morphism of toric varieties 
$$\stackT_{X} \xrightarrow{\Tilde{f}} \stackT_{Y}$$ 
which is equivariant with respect to a morphism of tori $T_{\NS,X} \rightarrow T_{\NS,Y}$. 

In \cite[theorem 9.5]{salberger_torsor}, Salberger constructs a canonical section 
$$X(K_v) \xrightarrow{\check{\psi}_X} \stackT_X(K_v) / T_{\NS,X}(K_v)_{\mathrm{cp}}$$ 
that, by construction, behaves functorially in $X$ with respect to toric morphisms. 
Here, $T_{\NS,X}(K_v)_{\mathrm{cp}}$ denotes the maximal compact subgroup of $T_{\NS,X}(K_v)$. 

The proposition then follows from the fact that $T(K_v)$ is dense in $X(K_v)$ and that, according to \cite[proposition 9.7]{salberger_torsor}, we have the equality 
$$||1(P)||_D = | \check{\psi}(P) |_D,$$ 
where the definition of $|-|_D$ is given in \cite[notations 9.6]{salberger_torsor}. 
Indeed, for a divisor $D$ on $Y$, we have
$$||1(f(P))||_D =  | \check{\psi_Y}(f(P)) |_D = | \overline{f}(\check{\psi_X}(P)) |_D = | \check{\psi_X}(P) |_{f^* D} = ||1(P)||_{f^*D}.$$

\end{proof}

We can describe the local height at the universal torsor $\stackT_{\Sigma}$. We use the functorial property \ref{fonctorialite_systeme_hauteur} of the notion of system of heights for toric morphism and we get that for any $v \in M_{\Q}$ and any $k \in \{ 1,..,t \}$ :
$$ H_{\stackT,v}(y,L_k) = \max\limits_{D \in \mathcal{M}_k} |X^D(y)|_v .$$

In fact using the description of the universal torsor given by the cartesian diagram \ref{catesian_diagram_very_ample_basis}, we can deduce an other description of the $\Z$-points of $\stackT_{\Sigma}$:

\begin{prop}\label{description_universal_torsor_toric_varieties}
    \ \vspace{0,1 cm}
    $$ \stackT_{\Sigma}(\Z) = \{ y \in \Z^{\Sigma(1)} \cap  \stackT_{\Sigma}(\Q)  | \forall \ i \in \{1,..,t \}, \ \gcd\limits_{D \in \mathcal{M}_i}(X^{D}(y)) = 1 \} .$$
    In particular, for any $y \in \stackT_{\Sigma}(\Z)$, we have for $1 \leqslant i \leqslant t$,
    $$ H_{L_i}(q(y)) = \max\limits_{D \in \mathcal{M}_i}(|X^{D}(y)|) = H_{\stackT,\infty}(y,L_i) .$$
\end{prop}
    
As $L_1,..,L_t$ form a basis of $\Pic(X)$, we have proven theorem \ref{simplification_of_the_height_to_univ_torsor}.

\subsubsection{The action of the Neron-Severi torus on each coordinate}\label{paragraph_action_Tns_coordonee}

In this paragraph, we shall describe the action of $T_{\NS}$ on $ \stackT_{\Sigma}$ seen as the universal torsor (see proposition \ref{universal_torsor_over_toric_variety}). It is induced by the opposite action of the one given by the group morphism $\pi^* : T_{\NS} \rightarrow T_{\Sigma(1)} $, dual to the morphism $\pi : \Div_T(X) \rightarrow \Pic(X)$ which maps $D$ to $[D]$.

We shall use the exponential group morphism:
\begin{align*}
    e^* : \Pic(X)&^{\vee}_{\R} \hookrightarrow \T_{\NS}(\R) \\
    &a \longmapsto e^a
\end{align*}such that for any $L \in \Pic(X)$, $\chi_L(e^a) = e^{a(L)} $ where $\chi_L \in X^*(T_{\NS})$ is the character associated to $L \in \Pic(X)$.

\begin{prop}\label{element_important_du_tore_NS}
    Let $a \in \Pic(X)^{\vee}_{\R}$. For any $\rho \in \Sigma(1)$, the action of $e^{a} \in T_{\NS}(\R)$ on the coordinate of $\stackT_{\Sigma}(\R)$ corresponding to the edge $\rho$ is given by multiplication by $e^{- \langle [D_{\rho}],a\rangle}$.
\end{prop}

\begin{proof}
    To any $T$-invariant divisor $D$, we can associate a character $\chi_{D} : T_{\Sigma(1)} \rightarrow \G_m$ which maps $y$ to $X^D(y)$. The composition $$\chi_{D} \circ \pi^* : T_{\NS} \rightarrow \G_m$$ corresponds to the character $\chi_{[D]}$ where $[D] = \pi(D) \in \Pic(X)$ is the class of $D \in \Div_T(X)$.
    Hence for any $\rho \in \Sigma(1)$, following proposition \ref{universal_torsor_over_toric_variety}, the action of $t \in T_{\NS}(\R)$ on the coordinate of $\stackT_{\Sigma}(\R)$ corresponding to the edge $\rho$ is given by multiplication by $\chi_{[D_{\rho}]}(t)^{-1}$ as $\chi_{D_{\rho}} \circ \pi^* = \chi_{[D_{\rho}]}$.
\end{proof}

\subsubsection{A natural covering of the real variety associated to our toric variety}

In this paragraph, $X(\R)$ is seen as a real manifold, endowed of its associated topology. Then we shall prove the following result on the local height $h_{\stackT,\infty}$ (see definition \ref{definition_local_height}):

\begin{prop}
    The restriction map given by $h_{\stackT,\infty}^{-1}(\{0\}) \rightarrow X(\R)$ is a finite covering.
\end{prop}

\begin{proof}
    Let $P \in X(\R)$ and let $U$ be an open subscheme of $X_{\R}$ such that $P \in U(\R)$ and that there exists a section $s : U \rightarrow \stackT_{\R}$ of $q : \stackT_{\R} \rightarrow X_{\R}$. Such section exists because by Hilbert's 90, we have $H^1_{\text{et}}(X_{\R},T_{\NS}) = H^1_{\text{Zar}}(X_{\R},T_{\NS})$. Hence we have a continuous section $$P \in U(\R) \longmapsto e^{h_{\stackT,\infty}(s(P))} . s(P)$$ of $h_{\stackT,\infty}^{-1}(\{0\}) \xrightarrow{p} X(\R)$ because, by proposition \ref{hauteur_locale_torseur_univ_et_action_de_tns}, we have $$H_{\stackT,\infty}(e^{h_{\stackT,\infty}(s(P))} . s(P),L) = e^{-h_{\stackT,\infty}(s(P))(L)} \cdot H_{\stackT,\infty}(s(P),L) = 1 $$for any $L \in \Pic(X)$. It gives using lemma \ref{hauteur_locale_torseur_univ_et_action_de_tns} a trivialisation of the covering as follows:
    \begin{align*}
        U(\R) &\times \ker(\log_{T_{\NS},\infty}) \rightarrow p^{-1}(U(\R)) \\
        &(P,t) \longmapsto t.e^{h_{\stackT,\infty}(s(P))}. s(P).
    \end{align*}
    As $\ker(\log_{T_{\NS},\infty}) \simeq \{-1,1\}^t$, it is a finite covering.
\end{proof}

\begin{cor}\label{corollaire_compacite}
    $h_{\stackT,\infty}^{-1}(\{0\}) $ is compact.
\end{cor}

\begin{proof}
    By \cite[chapitre I, §4, n°5, théorème 1]{bourbaki2016topologie}, the map $h_{\stackT,\infty}^{-1}(\{0\}) \rightarrow X(\R)$ is proper. Moreover, $X(\R)$ is compact for the real topology as it is a projective variety. Hence we get our corollary.
\end{proof}

\subsection{Measure over the universal torsors and Tamagawa number}

\subsubsection{Adelic measure and Tamagawa number  over a variety}

Let $V$ be a smooth, separated scheme which have geometrically connected fibers over $\Spec(\Z)$ and such that $V_{\Q}$ is a smooth, projective and geometrically integral variety over $\Q$. Let us recall here how an adelic metric $(||-||_v)_{v \in M_{\Q}} $ over $\w_{V}^{-1}$ defines a measure $\mu_V$ over $$V(A_{\Q}) = \prod\limits_{v \in M_{\Q}} V(\Q_v).$$ 

As in \cite[2.2.1]{peyre_duke}, we can define locally a measure over $V(\Q_v)$ by the relation $$ \Big\Vert \frac{\partial}{\partial x_{1,v}} \wedge..\wedge \frac{\partial}{\partial x_{n,v}} \Big\Vert_v d x_{1,v} .. d x_{n,v} $$where $(x_1,..,x_n) : \Omega \rightarrow \Q_v^n$ is a local system of coordinates defined on an open subset $\Omega \subset V(\Q_v)$ and $\frac{\partial}{\partial x_{1,v}} \wedge..\wedge \frac{\partial}{\partial x_{n,v}} $ is seen as a local section of $\w_V^{-1}$. This local measure does not depend on the choice of coordinates, therefore we can glue together these measures and we get a measure $\mu_{V,v}$ over $V(\Q_v)$.

By \cite[lemma 2.1.2]{peyre_duke}, we have that if $V$ has good reduction at every finite places and the norm is defined by the model (such hypotheses are verified by a split toric variety), then: $$ \mu_{V,v}(V(\Q_v)) = \frac{\sharp V(\F_v)}{p_v^{\dim(V)}} .$$In order to define
the Tamagawa number of our variety $V,$ we first define convergence factors. We set $\lambda_v = L_v(1,\Pic(\overline{V}))$ for $v$ a finite place and $\lambda_v = 1$ for $v$ is infinite.

\begin{definition}\label{tamaga_number_variety}
    The Tamagawa number of $V$ is defined as $\tau(V) = \mu_V(V(A_{\Q}))$ where $\mu_V = \prod\limits_{v \in M_{\Q}} \lambda_v^{-1} \mu_{V,v}$.
\end{definition}

\begin{remark}\label{tamagawa_number_toric_variety}
    In the case where $V$ is split, we have that for any $v \in M_{\Q}^0$, $$ \lambda_v = \left(1 - \frac{1}{p} \right)^t.$$We get that if moreover $V$ has a good reduction at every finite places:
    $$ \tau(V) = \mu_{V,\infty}(V(\R))  \prod\limits_{p \in \Prems} \left( 1 - \frac{1}{p} \right)^t  \frac{\sharp V(\F_p)}{p^{\dim(V)}} .$$
\end{remark}

Before defining a natural measure over universal torsor, let us recall \cite[definition 3.1.2]{peyre_zeta_height_function}, that one can define a natural measure over a split torus $T$:

\begin{definition}\label{mesure_sur_un_tore}
    Let $(\xi_1,..,\xi_t) $ be a basis of $X^*(T)_{\R}$ then the differential form $$ \w_{T} = \bigwedge\limits_{i = 1}^t \frac{d \xi_i}{\xi_i}$$does not depend on the choice of the basis, up to its sign. It is called the canonical form of $T$. This form defines a natural measure over $ T(\R)$.
\end{definition}

\subsubsection{Adelic measure over universal torsors}

In order to make the notations less cluttered, for the rest of this article, we shall write $\stackT = \stackT_{\Sigma}$ where $\stackT_{\Sigma}$ is the universal torsor defined in proposition \ref{universal_torsor_over_toric_variety} over the toric variety $X$. In this paragraph, we recall the construction from \cite[section 4.3.2]{Peyre_beyond_height} to define for any $v \in M_{\Q}$, a measure $\w_{\stackT,v}$ over $\stackT(\Q_v)$. We write $\widehat{T_{\NS}}$ for the sheaf of characters associated to the Neron-Severi torus.

%Using the universal property of $\stackT$, one can show that the pull-back of any pointed torsor over $(X,1)$ for a multiplicative group is trivial (see \cite[proposition 2.1.1]{colliot_sansuc_descente_variete_II}). 

Using the exact sequence $$ 1 \rightarrow \G_{m,X} \rightarrow q_* \G_{m,\stackT} \rightarrow \widehat{T_{\NS}} \rightarrow 0$$ associated to the universal torsor, which corresponds up to sign to the element  $[\stackT] \in H^1(X,T_{\NS}) $ seen as an element in $\Ext^1_X(\G_{m,X},T_{\NS})$ (see \cite[1.4.3]{ColliotTheleneSansuc1987}), we get that $H^0(\stackT,\G_m) = \Q^{\times}$ and that $\Pic(\stackT)$ is trivial. Moreover, by \cite[lemme 2.1.10]{methode_du_cercle_peyre}, $\w_{\stackT}$ is isomorphic to the pull-back of $\w_X$. Hence there exists a global section of $\w_{\stackT}$ corresponding to a trivialization of the line bundle so we get the following proposition \cite[proposition 4.24]{Peyre_beyond_height}:

\begin{prop}
   Up to multiplication by a scalar, there exists a unique volume form over $\stackT$.
\end{prop}

We follow \cite[construction 4.25]{Peyre_beyond_height} in order to describe locally the measure $\w_{\stackT,v}$ over $\stackT(\Q_v)$. Let us choose a non zero section $\w$ of $\w_{\stackT}$. For any place $v \in M_{\Q}$, the expression
$$ \left| \left\langle \w , \frac{\partial}{\partial x_1} \wedge .. \wedge \frac{\partial}{\partial x_m} \right\rangle \right|_v d x_{1,v}.. d x_{m,v} $$defined a local measure. Gluing this local measures, we get a measure $\w_{\stackT,v}$ over $\stackT(\Q_v)$. 

\begin{remark}\label{intrepretation_mesure_universal_torsor}
    In our case, recall that $X$ is a smooth, proper and split toric variety and $\stackT$ an open subset of $\affine^{\Sigma(1)}$, the previous construction gives us that $\w_{\stackT,v}$ is the restriction of the normalised Haar measure $d x_v$ over $\Q_v^{\Sigma(1)}$ to the open subset $\stackT(\Q_v)$ if $v \in M_{\Q}^0$ and is the restriction of the natural Lesbesgue measure $d x_{\infty}$ over $\R^{\Sigma(1)}$ to $\stackT(\R)$. 
\end{remark}

We shall compute $\w_{\stackT,v}(\stackT(\Z_v))$ for $v \in M_{\Q}^0$:

\begin{theorem}\label{mesure_torseur_univ_tamagawa_number_toric}
    For any $v \in M_{\Q}^0$, we have:
    $$\w_{\stackT,v}(\stackT(\Z_v)) = \frac{\sharp \stackT(\F_v)}{p_v^{\dim(\stackT)}} =  \left( 1 - \frac{1}{p_v} \right)^t . \frac{\sharp X( \F_v) }{p_v^{\dim(X)}} .$$
\end{theorem}

\begin{proof}
    Because $\w_{\stackT,v}$ is the restriction of the normalised Haar measure $d x_v$ over $\Q_v^{\Sigma(1)}$ to the open subset $\stackT(\Q_v)$ , we have $\w_{\stackT,v}(\stackT(\Z_v)) = \frac{\sharp \stackT(\F_v)}{p_v^{\dim(\stackT)}}$. Using that $T_{\NS}(\F_v)$ acts freely on $\stackT(\F_v)$ and that $X(\F_v) = \stackT(\F_v) / T_{\NS}(\F_v)$,  we can conclude the proof.
\end{proof}

To end this section, we shall compute $\w_{\stackT,\infty}(\mathcal{D}) = \Vol(\mathcal{D})$ where $$\mathcal{D} = \{ y \in \stackT(\R) \mid h_{\stackT,\infty}(y) \in \mathrm{D} \}$$ for a compact subset $\mathrm{D} \subset \Pic(X)^{\vee}_{\R}$. We shall connect it to $\nu(\mathrm{D})$ where $\nu$ is the measure over $\Pic(X)^{\vee}_{\R}$ defined as in \ref{mesure_picard_group}. We shall need the following lemma which was stated in \cite[4.4.1]{peyre_zeta_height_function}:

\begin{lemma}
    Let $P \mapsto y_P$ be any measurable section of $\stackT(\R) \rightarrow X(\R)$. For any function $f : X(\R) \rightarrow \R$ compactly supported and measurable, the function:
    $$ P \longmapsto \int\limits_{T_{\NS}(\R)} | \chi_{\w_{X}^{-1}}(t) |_{\infty}^{-1} f(t.y_P) \w_{T_{\NS}}(dt)$$ does not depend on the choice of the section $P \mapsto y_P$ and we have:
    $$ \int\limits_{\stackT(\R)} f(y) \w_{\stackT,\infty}(dy) = \int\limits_{X(\R)} \mu_{X,\infty}(dP) \int\limits_{T_{\NS}(\R)} | \chi_{\w_{X}^{-1}}(t) |_{\infty}^{-1} f(t.y_P) \w_{T_{\NS}}(dt) .$$
\end{lemma}

\begin{proof}
    The first assertion is a consequence of the invariance of the Haar measure (see \cite[chapitre VII, §2, n°2, proposition 1]{bourbaki2007integration}). To show the second assertion, one must remark that in \cite[section 5]{salberger_torsor}, Salberger shows that if you define the measure $$\w_{\text{Sal},\infty}(dy) = \frac{\w_{\stackT,\infty}(dy)}{H_{\stackT,\infty}(y,\w_X^{-1})} ,$$ it is a $T_{\NS}(\R)$-invariant measure and the quotient measure $\frac{\w_{\text{Sal},\infty}}{\w_{T_{\NS}}} $ is $\mu_{X,\infty}$. Hence by \cite[chapitre VII, §2, n°2, proposition 4]{bourbaki2007integration}, we have that for any $f : X(\R) \rightarrow \R$ compactly supported function, we may choose a section $P \mapsto y_P$ such that $$H_{\stackT,\infty}(y_P,\w_{X}^{-1}) = 1$$for any $P$, and we get:
    \begin{align*}
      \int\limits_{\stackT(\R)} f(y) \w_{\stackT,\infty}(dy)  &= \int\limits_{\stackT(\R)} f(y) H_{\stackT,\infty}(y,\w_{X}^{-1}) \w_{\text{Sal},\infty}(dy) \\
      &= \int\limits_{X(\R)} \mu_{X,\infty}(dP) \int\limits_{T_{\NS}(\R)} H_{\stackT,\infty}(t.y_P,\w_{X}^{-1}) f(t.y_P) \w_{T_{\NS}}(dt) .
    \end{align*}Using that $ H_{\stackT,\infty}(t.y_P,\w_{X}^{-1}) = | \chi_{\w_{X}^{-1}}(t) |^{-1}_{\infty} . H_{\stackT,\infty}(y_P,\w_{X}^{-1}) $, we can conclude the proof.
\end{proof}

\begin{prop}\label{analysis_measure_univ_torsor_at_infty}
   $$ \w_{\stackT,\infty}(\mathcal{D}) = \Vol(\mathcal{D}) = 2^t .\mu_{X,\infty}(X(\R)) . \nu(\mathrm{D}) . $$
\end{prop}

\begin{proof}
   Applying the previous lemma for $f = \mathbf{1}_{\mathcal{D}}$, we get:
    $$\w_{\stackT,\infty}(\mathcal{D}) =  \int\limits_{P \in X(\R)}  \mu_{X,\infty}(dP) \int\limits_{T_{\NS}(\R)} | \chi_{\w_{X}^{-1}}(t) |_{\infty}^{-1} \mathbf{1}_{\log_{T_{\NS},\infty}^{-1}(\mathrm{D})}(t) \w_{T_{\NS}}(t) $$where for each $P \in X(\R)$, we have choosen $y_P \in \stackT_P(\R)$ such that $h_{\stackT,\infty}(y_P) = 0$.
    Then using that we have the following exact sequence of locally compact groups:
    $$0 \rightarrow (\{ \pm 1 \}^t,\delta_{\text{compt}}) \rightarrow (T_{\NS}(\R),\w_{T_{\NS}}) \xrightarrow{\log_{T_{\NS},\infty}} (\Pic(X)^{\vee}_{\R},da) \rightarrow 0 $$where $da$ is the Lebesgue measure over $\Pic(X)^{\vee}_{\R}$, we get:
    \begin{align*}
        \w_{\stackT,\infty}(\mathcal{D}) &= 2^t . \int\limits_{P \in X(\R)}  \mu_{X,\infty}(dP) \int\limits_{\Pic(X)^{\vee}_{\R}} e^{\langle \w_{X}^{-1},a \rangle} \mathbf{1}_{\mathrm{D}}(a) da \\
        &= 2^t . \mu_{X,\infty}(X(\R)) . \nu(\mathrm{D}) .
    \end{align*}
   %Then to conclude we just have to see that because $\mu_{X,\infty}$ has locally a density compared with the Lebesgue measure, as $X(\R) \setminus T(\R)$ is of codimension $\geqslant 1$, $\mu_{X,\infty}(X(\R) \setminus T(\R)) = 0$. Hence:
  % $$ \mu_{X,\infty}(T(\R)) = \mu_{X,\infty}(X(\R))$$

\end{proof}

\subsection{The Möbius function associated to the universal torsor}

For this section, we follow the notations of Salberger \cite[Notations 11.8]{salberger_torsor} in order to introduce the Möbius function associated to the universal torsor $\stackT$ over $X$.

\begin{Notations}
\begin{enumerate}
    \item $\mathbf{1}_{\stackT(\Z)} : (\N^*)^{\Sigma(1)} \rightarrow \{0,1\}$ is the indicator function of $\stackT(\Z) \cap (\N^*)^{\Sigma(1)}$, that is to say we have the following equivalence: $$\mathbf{1}_{\stackT(\Z)}(e) = 1 \Leftrightarrow  \gcd(y^{\Tilde{\sigma}}, \sigma \in \Sigma_{\max}) = 1 .$$
    \item for $d \in (\N^*)^{\Sigma(1)}$, we write $\mathbf{1}_{d.\Z^{\Sigma(1)}} : (\N^*)^{\Sigma(1)} \rightarrow \{0,1\}$ the indicator function of $\bigoplus\limits_{\rho \in \Sigma(1)} d_{\rho}.\Z \cap (\N^*)^{\Sigma(1)}$, that is to say we have the following equivalence: $$\mathbf{1}_{d.\Z^{\Sigma(1)}}(e) = 1 \Leftrightarrow  d \mid e $$where $d \mid e$ means $d_{\rho} \mid e_{\rho}$ for any $\rho \in \Sigma(1)$. 
\end{enumerate}
\end{Notations}

In  \cite[definition 11.9]{salberger_torsor}, Salberger defined a Möbius function:
\begin{definition}\label{mobius_function_definition}
   $$\mu : (\N^*)^{\Sigma(1)} \rightarrow \Z$$such that $$ \mathbf{1}_{\stackT(\Z)} = \sum\limits_{d \in (\N^*)^{\Sigma(1)}} \mu(d) \mathbf{1}_{d.\Z^{\Sigma(1)}} .$$ 
\end{definition}

For a prime number $p$, we denote by $\sum\limits_{d}^{(p)}$ the sum over the $\Sigma(1)$-tuple of the form $d = (p^{e_{\rho}})_{\rho \in \Sigma(1)}$ where $(e_{\rho})_{\rho \in \Sigma(1)}$ belongs to $\N^{\Sigma(1)}$. We also define $\Pi(d) = \prod\limits_{\rho \in \Sigma(1)} d_{\rho}$ for any $d \in (\N^*)^{\Sigma(1)}$. Salberger showed (see \cite[11.14]{salberger_torsor}) the following result:

\begin{prop}\label{value_local_mesure_mobius_function}
For any prime number $p$,
    $$\w_{\stackT,p}(\stackT(\Z_p)) =  \sum\limits_{d}^{(p)} \frac{\mu(d)}{\Pi(d)}$$
\end{prop}

% Si I est une partie de cardinal f qui n'est pas contenue dans aucun cöne. Alors c'est une collection primitive par minimalité de f. De plus toute collection primitive n'est contenue dans aucun cone d'où l'affirmation f = \min \sharp C

Recall that we denote by $l$ be the smallest positive integer such that there exists $l$ rays of $\Sigma$ not contained in a cone of $\Sigma$. We have the following result from \cite[lemma 11.15]{salberger_torsor}:

%By definition \ref{def_primitive_collection}, $f$ is the minimum of $\sharp C$ where $C$ belongs to the set of primitive collections. Note that $f \geqslant 2$\footnote{comment montrer que $\sharp C \geqslant \dim(X) + 1$ pour $C$ une collection primitive ?}. 

\begin{prop}\label{value_global_mesure_mobius_fonction}
    The Eulerian product over every prime numbers:
    $$\prod\limits_p \sum\limits_{d}^{(p)} \frac{|\mu(d)|}{\Pi(d)^s}$$is absolutely convergent for $\Re(s) > \frac{1}{l}$. In the same way, $$\sum\limits_{d \in (\N^*)^{\Sigma(1)}} \frac{| \mu(d) |}{\Pi(d)^s} $$ is convergent and is equal to $$\prod\limits_p \sum\limits_{d}^{(p)} \frac{|\mu(d)|}{\Pi(d)^s}$$for $\Re(s) > \frac{1}{l}$.
\end{prop}

Using this proposition, one can deduce (see \cite[lemma 11.19]{salberger_torsor}) the following lemma:
\begin{lemma}\label{majoration_fonction_de_mobius}
    Let $\epsilon \in ]0,1 - \frac{1}{l}[$. We have the following upper bounds:
    \begin{enumerate}
        \item $\sum\limits_{\Pi(d)\leqslant Y} |\mu(d)| = O\left( Y^{\frac{1}{l} + \varepsilon}\right)$
        \item $\sum\limits_{\Pi(d)\geqslant Y} \frac{| \mu(d) |}{\Pi(d)} = O\left( \frac{1}{Y^{1 - \frac{1}{l} - \epsilon}} \right).$
    \end{enumerate}
\end{lemma}

\section{Lift of the number of rational points of the toric variety to its universal torsor}

The method of lifting the number of points to universal torsors have been used in  a particular case in \cite{peyre_duke}, then extended to toric varieties in \cite[section 11]{salberger_torsor} and generalised in \cite{methode_du_cercle_peyre},\cite{Peyre_beyond_height}.

%We can found the same heuristic in recent articles, for example in order to count cubic extensions of $\Q$ (\cite{bhargava_counting_cubic_field}).

We fix $u \in (\cone^{\vee})^{\circ}$ and $\mathrm{D} \subset \Pic(X)^{\vee}_{\R}$ a finite union of compact polyhedrons. Let $B > 1$ and write $\mathrm{D}_B = \mathrm{D} + \log(B) u $. We shall use that the map
$$ q : \stackT \rightarrow X$$is surjective. Using the lift of rational points to the universal torsor, we have by corollary \ref{simplification_of_the_height_to_univ_torsor}:
\begin{align*}
    \sharp X(\Q)_{h \in \mathrm{D}_B} &= \frac{1}{2^t} \sharp \{ y \in \stackT(\Z) \mid h_{\stackT,\infty}(y) \in \mathrm{D}_B \} \\
    &= \frac{1}{2^t} \sharp \left( \stackT(\Z) \right)_{h_{\stackT,\infty} \in \mathrm{D}_B}
\end{align*}

Using the Möbius inversion, we shall see that estimating $\sharp \left( \stackT(\Z) \right)_{h_{\stackT,\infty} \in \mathrm{D}_B}$ reduces to estimating:
$$\sharp \Bigl( \Z^{\Sigma(1)} \cap \stackT(\R) \Bigr)_{h_{\stackT,\infty} \in \mathrm{D}_B} .$$

\subsection{Estimation of the error term between the number of points of a lattice in a bounded domain and the volume of the domain}

In what follows, $\Vol$ means the usual euclidean volume given by the Lebesgue measures. The aim of this section is to relate $\sharp \Bigl( \Z^{\Sigma(1)} \cap \stackT(\R) \Bigr)_{h_{\stackT,\infty} \in \mathrm{D}_{B}}$ and the volume of $\mathcal{D}_B = \{ y \in \stackT(\R) \mid h_{\stackT,\infty}(y) \in \mathrm{D}_{B}\}$. To estimate the difference between this two terms, we shall use a result of geometry of numbers due to Davenport \cite{davenport_geometry_numbers}. The statement we shall use is from \cite[proposition 24]{bhargava_counting_cubic_field}:

\begin{prop}\label{bhargava_geometrie_nombre}
Let $\mathcal{D}$ be a bounded domain,  semi-algebraic of $\R^m$ which is defined by at most $k$ polynomial inequalities each having degree at most $l$.

Then the number of integral lattice points of $\Z^m$ contained in $\mathcal{D}$ is :

$$\Vol(\mathcal{D}) + O\bigl( \max(1,\Vol(\proj_I(\mathcal{D})) | I \in \mathfrak{P}(\{1,..,m\}),  \sharp 1 \leqslant I < m) \bigr)$$where $\Vol(\proj_I(\mathcal{D}))$ means the euclidean volume of the projection of $\mathcal{D}$ via 
\begin{align*}
    \proj_I : &\R^m \rightarrow \R^{ I} \\
    (x_1,..,&x_m) \longmapsto (x_i)_{i \in I} .
\end{align*}
The constant in the error term depends only on $m$, $k$, and $l$.

\end{prop}

Each time we shall use this proposition in this article, the constant $m$, $k$ and $l$ will not vary. We shall apply this proposition to the semi-algebraic bounded set $\Psi(\mathcal{D}_1)$ where:
\begin{enumerate}
    \item $\mathcal{D}_1 = \{ y \in \stackT(\R) \mid h_{\stackT,\infty}(y) \in \mathrm{D}\} $
    \item $\Psi$ is the multiplication by an element $(a_{\rho})_{\rho \in \Sigma(1)} \in (\R_+^{\times})^{\Sigma(1)}$. In other words:
    \begin{align*}
        \Psi : \ &\R^{\Sigma(1)} \rightarrow \R^{\Sigma(1)} \\
        &(z_{\rho})_{\rho \in \Sigma(1)} \longmapsto \left( a_{\rho} . z_{\rho} \right)_{\rho \in \Sigma(1)}
    \end{align*}
\end{enumerate}Hence  we shall have $m = \sharp \Sigma(1)$, $k$ depends only on the choice of $\mathrm{D}$, the finite union of compact polyhedrons and $l$ depends on the choice of $\mathrm{D}$ and on the maximum of the degree of the monomial $X^D$ for $D \in \mathcal{M}_k$ and $k \in \{1,..,t\}$ (see notations from \ref{base_section_globale}). By this remark, applying the previous proposition gives that there exists $C_1 > 0$ and $C_2 > 0$, which does not depend on the choice of $\Phi$, such that:
\begin{align*}
    &\left| \sharp ( \Z^{\Sigma(1)} \cap \Phi(\mathcal{D}_1) ) - \left(\prod\limits_{\rho \in \Sigma(1)} a_{\rho} \right).\Vol(\mathcal{D}_1) \right| \\
    &\leqslant C_1 + C_2 .\max\limits_{I \subsetneq \Sigma(1) \mid I \neq \varnothing} \left( \left(\prod\limits_{\rho \in I} a_{\rho} \right) . \Vol(\proj_I(\mathcal{D}_1)) \right)
\end{align*}

For the rest of the article, we shall suppose that $B \geqslant 1$. Recall we have chosen a vector $u \in \left(\cone^{\vee}\right)^{\circ}$ and defined $\mathrm{D}_{B} = \mathrm{D} + \log(B) u$. 

\subsubsection{An upper bound on the coordinates}

The key point to estimate the error term is the following statement:

\begin{prop}\label{crucial_majoration}
    There exists a constant $C_{\mathrm{D}} > 0$ (depending on the compact subset $\mathrm{D}$) such that for any $y \in \mathcal{D}_{B}$, for any $\rho \in \Sigma(1)$, we have:
    
    $$ |y_{\rho}| \leqslant C_{\mathrm{D}} B^{ \langle [D_{\rho}],u \rangle } .$$ 
\end{prop}

From now on, we shall write by $B^{u}$ the unique element of $T_{\NS}(\R)$ given by $\exp(\log(B).u)$. Using lemma \ref{hauteur_locale_torseur_univ_et_action_de_tns}, we get that $$\mathcal{D}_{B} = B^{-u}. \mathcal{D}_1 .$$

Then to prove the previous proposition, we have to recall that the action of $B^{-u}$ on the coordinate associated to $\rho \in \Sigma(1)$ is given by multiplication by $B^{ \langle [D_{\rho}],u\rangle }$ by proposition \ref{element_important_du_tore_NS} and we can conclude using the following lemma:

\begin{lemma}\label{majoration_crucial_dans_un_compact}
    There exists a constant $C_{\mathrm{D}} > 0$ (depending on the compact subset $\mathrm{D}$) such that for any $y \in \mathcal{D}_{1}$, for any $\rho \in \Sigma(1)$, we have:
    
    $$ |y_{\rho}| \leq C_{\mathrm{D}} .$$
\end{lemma}

\begin{proof}
    We can define a surjective continuous map:
    \begin{align*}
        h_{\stackT,\infty}^{-1}(\{0\}) &\times \mathrm{D} \rightarrow \mathcal{D} \\
        (y&,v) \longmapsto e^{-v} . y
    \end{align*}
    By corollary \ref{corollaire_compacite}, $h_{\stackT,\infty}^{-1}(\{0\})$ is compact. Hence $\mathcal{D}_1$ is the image of a compact by a continuous map so it is compact. There exists a constant $C > 0$ such that $\mathcal{D}_1 \subset B_{|-|_{\infty}}(0,C)$.
\end{proof}

We get the following corollary of our theorem:

\begin{cor}
    For any $I$ non empty subset of $\Sigma(1)$, there exists a constant $C > 0$ such that:
    $$\Vol(\proj_I(\mathcal{D}_{B})) \leq C B^{ \langle \sum\limits_{\rho \in I} [D_{\rho}] , u \rangle } .$$
\end{cor}

We finish this paragraph by recalling the value of the volume $\Vol(\mathcal{D}_{B}) $ and by estimating the cardinal of $\sharp \Bigl( \Z^{\Sigma(1)} \cap \stackT(\R) \Bigr)_{h_{\stackT,\infty} \in \mathrm{D}_{B}}$.

\begin{theorem}
    $$\Vol(\mathcal{D}_{B}) =  \Vol(\mathcal{D}_1) B^{ \langle \w_{X}^{-1} , u \rangle } = 2^t. \mu_{X,\infty}(X(\R)). \nu(\mathrm{D}_1) B^{ \langle \w_{X}^{-1} , u \rangle } . $$
\end{theorem}

\begin{proof}
    By remark \ref{intrepretation_mesure_universal_torsor}, the Lebesgue measure over $\R^{\Sigma(1)}$ restricted to $\stackT(\R)$ is $\w_{\stackT,\infty}$. So the theorem is a consequence of proposition \ref{analysis_measure_univ_torsor_at_infty}.
\end{proof}

Recall that as we take $u \in (\cone^{\vee})^{\circ}$  i.e. for any $L \in \cone$, $\langle L,u \rangle > 0 $, then we have that: 

$$\max\limits_{ I \subsetneq \Sigma(1) \mid I \neq \varnothing } \Vol(\proj_I(\mathcal{D}_B)) = O\left( \frac{B^{\langle \w_X^{-1},u\rangle}}{B^{ \min\limits_{\rho \in \Sigma(1)} \langle [D_{\rho}],u\rangle } } \right)$$which is negligible compared to $\Vol(\mathcal{D}_B)$ when $B$ grows. We can then deduce the following theorem using proposition \ref{bhargava_geometrie_nombre}:

\begin{theorem}\label{estimation_lattice}
    There exists a constant $C > 0$ such that:
    $$\Big| \sharp \Bigl( \Z^{\Sigma(1)} \cap T_{\Sigma(1)}(\R) \Bigr)_{h_{\stackT,\infty} \in \mathrm{D}_{B}}  - 2^t . \mu_{X,\infty}(X(\R)). \nu(\mathrm{D}_1) B^{ \langle \w_{X}^{-1} , u \rangle } \Big| \leqslant C \frac{B^{\langle \w_X^{-1},u\rangle}}{B^{ \min\limits_{\rho \in \Sigma(1)} \langle [D_{\rho}],u\rangle } } $$
\end{theorem}

\subsubsection{Upper bound for the number of points in the complementary of the torus}

Applying lemma \ref{majoration_crucial_dans_un_compact}, we have that there exists $N > 1$ such that $\mathcal{D}_1 \subset \prod\limits_{\rho \in \Sigma(1)} [-N,N]$. Hence $\mathcal{D}_B \subset  \prod\limits_{\rho \in \Sigma(1)} \left[ -N.B^{\langle [D_{\rho}] , u \rangle},N.B^{\langle [D_{\rho}] , u \rangle} \right] $, so we have for any $\rho \in \Sigma(1)$,
\begin{align*}
    \sharp\left(\Z^{\Sigma(1)} \cap \stackT(\R) \cap \{y_{\rho} = 0 \} \right)_{h_{\stackT,\infty} \in \mathrm{D}_B} &\leqslant \prod\limits_{\rho' \in \Sigma(1)-\{\rho\}} \left( 2.\lfloor N.B^{\langle [D_{\rho}] , u \rangle} \rfloor + 1 \right) \\
    &\leqslant \prod\limits_{\rho' \in \Sigma(1)-\{\rho\}} \left( 2 N.B^{\langle [D_{\rho}] , u \rangle} + 1 \right)
\end{align*}Hence there exists a constant $C > 0$ such that:
$$\sharp\left(\Z^{\Sigma(1)} \cap \stackT(\R) \cap \{y_{\rho} = 0 \} \right)_{h_{\stackT,\infty} \in \mathrm{D}_B} \leqslant C. B^{\langle \sum\limits_{\rho \neq \rho'} [D_{\rho'}] , u \rangle} .$$We get the following theorem:

\begin{theorem}\label{number_of_points_complementary_of_the_torus_universal_torsor}
    $$ \sharp\left(\Z^{\Sigma(1)} \cap (\stackT(\R) \setminus T_{\Sigma(1)}(\R)) \right)_{h_{\stackT,\infty} \in \mathrm{D}_B} = O\left( \frac{B^{\langle \w_{X}^{-1} , u \rangle}}{B^{\min\limits_{\rho \in \Sigma(1)} \langle [D_{\rho}] , u \rangle }} \right)$$
\end{theorem}

Applying the previous theorem and theorem \ref{estimation_lattice}, we get that:

\begin{cor}
    There exists a constant $C_{\mathrm{D}} > 0$ such that:
    $$\Big| \sharp \Bigl( \Z^{\Sigma(1)} \cap \stackT(\R) \Bigr)_{h_{\stackT,\infty} \in \mathrm{D}_{B}}  - 2^t . \mu_{X,\infty}(X(\R)). \nu(\mathrm{D}_1) B^{ \langle \w_{X}^{-1} , u \rangle } \Big| \leqslant C_{\mathrm{D}} \frac{B^{\langle \w_X^{-1},u\rangle}}{B^{ \min\limits_{\rho \in \Sigma(1)} \langle [D_{\rho}],u\rangle } } .$$

\end{cor}

As $\sharp\left(X(\Q) \setminus T(\Q)) \right)_{h \in \mathrm{D}_B} \leqslant \sharp\left(\Z^{\Sigma(1)} \cap (\stackT(\R) \setminus T_{\Sigma(1)}(\R)) \right)_{h_{\stackT,\infty} \in \mathrm{D}_B}$, we can also deduce the following corollary:

\begin{cor}\label{number_of_points_complementary_of_the_torus_variety}
    $$ \sharp\left(X(\Q) \setminus T(\Q)) \right)_{h \in \mathrm{D}_B} = O\left( \frac{B^{\langle \w_{X}^{-1} , u \rangle}}{B^{\min\limits_{\rho \in \Sigma(1)} \langle [D_{\rho}] , u \rangle }} \right) .$$
\end{cor}

\subsection{Extension of our estimation in order to apply the Möbius inversion}

Let $d \in (\N^*)^{\Sigma(1)}$ and write $\Lambda_d = \bigoplus\limits_{\rho \in \Sigma(1)} d_{\rho}.\Z$. We shall denote $N_d(B)$ the cardinal of the set $\Big( \Lambda_d \cap T_{\Sigma(1)}(\R) \Big)_{h_{\stackT,\infty} \in \mathrm{D}_{B}}$. We want to extend the estimation made in theorem \ref{estimation_lattice} to $N_d(B)$.

Denote by $\psi : \R^{\Sigma(1)} \rightarrow \R^{\Sigma(1)}$ the diagonal endomorphism given by multiplication by $d_{\rho}$ the coordinate associated to the edge $\rho$. Recall that we have $\covol(\Lambda_d) = \Pi(d)$. We can interpret $N_d(B)$ as follows:
    \begin{align*}
        N_d(B) &= \sharp \left\{ y \in \Z^{\Sigma(1)} \mid h_{\stackT,\infty}\left( \psi(y)  \right) \in \mathrm{D}_B \text{ and } \psi(y) \in T_{\Sigma(1)}(\R) \right\} \\
        &= \sharp \left\{ y \in \Z^{\Sigma(1)} \mid h_{\stackT,\infty}\left( B^{u}. \psi(y) \right) \in \mathrm{D} \text{ and } B^{u}. \psi(y) \in T_{\Sigma(1)}(\R) \right\} \\
        %&= \sharp \left\{ y \in \Z^{\Sigma(1)} \mid h_{\stackT,\infty}\left( B^{-u}. \varphi(y) \right) \in \mathrm{D} \text{ and } y \in T_{\Sigma(1)}(\R) \right\}
    \end{align*}
We write $\mathcal{D}_1 = \{ y \in T_{\Sigma(1)}(\R) \mid h_{\stackT,\infty}(y) \in \mathrm{D} \}$. We can then deduce that $$N_d(B) = \sharp \left( \Z^{\Sigma(1)} \cap \Psi(\mathcal{D}_1) \right)$$ where $\Psi(z) = \psi^{-1} ( B^{-u}. z ) $. Using this description, we begin by giving an upper bound of $N_d(B)$:

\begin{prop}\label{upper_bound_N_d(B)}
    For any $B > 1$, we have:
    $$N_d(B) = O\left( \frac{B^{\langle \w_{X}^{-1} , u \rangle}}{\Pi(d)} \right) .$$
\end{prop}

\begin{proof}
    Again by lemma \ref{majoration_crucial_dans_un_compact}, there exists $N > 1$ such that $\mathcal{D}_1 \subset \prod\limits_{\rho \in \Sigma(1)} \left( [-N,N] - \{0\} \right)$. Hence $\Psi(\mathcal{D}_1) \subset  \prod\limits_{\rho \in \Sigma(1)} \left( \left[ -N.\frac{B^{\langle [D_{\rho}] , u \rangle}}{d_{\rho}},N.\frac{B^{\langle [D_{\rho}] , u \rangle}}{d_{\rho}} \right] - \{0\} \right)$. So we have that:
    $$N_d(B) = \sharp(\Z^{\Sigma(1)} \cap \Psi(\mathcal{D}_1)) \leqslant 2^{\Sigma(1)}.\prod\limits_{\rho \in \Sigma(1)} \Big\lfloor N.\frac{B^{\langle [D_{\rho}] , u \rangle}}{d_{\rho}} \Big\rfloor \leqslant (2.N)^{\sharp \Sigma(1)} . \frac{B^{\langle \w_X^{-1},u\rangle}}{\Pi(d)} .$$
\end{proof}

\begin{remark}\label{remark_nulitte_N_d(B)}
    One may remark using the previous proof that $N_d(B) = 0$ when there exists $\rho \in \Sigma(1)$ such that $N . B^{\langle [D_{\rho}] , u \rangle} < d_{\rho}$.
\end{remark}

Now we shall later need a better estimation of $N_d(B)$ when $\Pi(d) \leqslant B^{\min\limits_{\rho \in \Sigma(1)} \langle [D_{\rho}] , u \rangle}$. It is given by the following theorem:

\begin{theorem}\label{estimation_N_d(B)}
    There exists a constant $C_{\mathrm{D}} > 0$ (depending on the choice of $\mathrm{D})$ such that:
    $$\left|  N_d(B) - \frac{1}{\Pi(d)}.2^t.\mu_{X,\infty}(X(\R)). \nu(\mathrm{D}_1) B^{ \langle \w_{X}^{-1} , u \rangle } \right| \leqslant C_{\mathrm{D}} .\frac{B^{\langle \w_X^{-1},u\rangle}}{B^{ \min\limits_{\rho \in \Sigma(1)} \langle [D_{\rho}],u\rangle } } .$$
\end{theorem}

\begin{proof}
    Denote by $\Psi_I : \R^I \rightarrow \R^I$ the morphism which maps $$(x_{\rho})_{\rho \in I} \mapsto \left( \frac{B^{\langle [D_{\rho}] , u \rangle}}{d_{\rho}} . x_{\rho} \right)_{\rho \in I} .$$Then we may remark that:
  $$ \proj_I \circ \Psi = \Psi_I \circ \proj_I .$$Using this we get that there exists $C_{\mathrm{D}} > 0$ (which depends only on the compact subset $\mathrm{D}$) such that for any non empty set $I \subsetneq \Sigma(1) $:
   $$ \Vol(\proj_I(\Psi(\mathcal{D})) \leqslant C_{\mathrm{D}} . \frac{B^{\langle \sum\limits_{\rho \in I} [D_{\rho}] , u \rangle}}{\prod\limits_{ \rho \in I} d_{\rho}} \leqslant C_{\mathrm{D}}.\frac{B^{\langle \w_X^{-1},u\rangle}}{B^{ \min\limits_{\rho \in \Sigma(1)} \langle [D_{\rho}],u\rangle } }.$$Applying proposition \ref{bhargava_geometrie_nombre} finishes the proof of the lemma.
   
\end{proof}

%In order to apply Möbius inversion, we shall work with $Z_d(B) = \sharp( \Lambda_d \cap \stackT(\R) )_{h_{\stackT,\infty} \in \mathrm{D}_{B}}$. We need the following upper bound:

%\begin{prop}
    %$$ \sharp\left(\Lambda_d \cap (\stackT(\R) \setminus T_{\Sigma(1)}(\R)) \right)_{h_{\stackT,\infty} \in \mathrm{D}_B} = O\left( \frac{B^{\langle \w_{X}^{-1} , u \rangle}}{B^{\min\limits_{\rho \in \Sigma(1)} \langle [D_{\rho}] , u \rangle }} \right)$$
%\end{prop}

%\begin{proof}
    %This is a consequence of \ref{number_of_points_complementary_of_the_torus_universal_torsor} because $\Lambda_d \subset \Z^{\Sigma(1)}$.
%\end{proof}

%Then we have the following corollary:

%\begin{cor}\label{estimation_Z_d(B)}
    %There exists a constants $C > 0$ (depending on the choice of $\mathrm{D})$ such that:
    %$$\left|  Z_d(B) - \frac{1}{\Pi(d)}.2^t.\mu_{X,\infty}(X(\R)). \nu(\mathrm{D}_1) B^{ \langle \w_{X}^{-1} , u \rangle } \right| \leqslant C .\frac{B^{\langle \w_X^{-1},u\rangle}}{B^{ \min\limits_{\rho \in \Sigma(1)} \langle [D_{\rho}],u\rangle } } .$$
%\end{cor}

\subsection{The Möbius inversion}

We recall that we take $\mathrm{D}$ a finite union of compact polyhedrons of $\Pic(X)^{\vee}_{\R}$, $u \in (\cone^{\vee})^{\circ}$ and $B > 1$.

\begin{prop}
    $$ \sharp( \stackT(\Z) \cap T_{\Sigma(1)}(\R) )_{h_{\stackT,\infty} \in \mathrm{D}_B} = \sum\limits_{d \in (\N^*)^{\Sigma(1)}} \mu(d) . N_d(B)  $$
\end{prop}

\begin{proof}
    Firstly, by remark \ref{remark_nulitte_N_d(B)}, the sum is finite. Now using the definition of \ref{mobius_function_definition}, we have:
    \begin{align*}
        \sharp( \stackT(\Z) \cap T_{\Sigma(1)}(\R) )_{h_{\stackT,\infty} \in \mathrm{D}_B} &= 2^{\sharp \Sigma(1)}. \sum\limits_{y \in (\N^*)^{\Sigma(1)}} \mathbf{1}_{h_{\stackT,\infty}^{-1}(\mathrm{D}_B)}(y) . \mathbf{1}_{\stackT(\Z)}(y) \\
        &= 2^{\sharp \Sigma(1)}. \sum\limits_{y \in (\N^*)^{\Sigma(1)}} \mathbf{1}_{h_{\stackT,\infty}^{-1}(\mathrm{D}_B)}(y) . \sum\limits_{d \in (\N^*)^{\Sigma(1)}} \mu(d) . \mathbf{1}_{d.\Z^{\Sigma(1)}}(y) \\
        &= \sum\limits_{d \in (\N^*)^{\Sigma(1)}} \mu(d) . \left( 2^{\sharp \Sigma(1)}. \sum\limits_{y \in (\N^*)^{\Sigma(1)}} \mathbf{1}_{h_{\stackT,\infty}^{-1}(\mathrm{D}_B)}(y) . \mathbf{1}_{d.\Z^{\Sigma(1)}}(y) \right) \\
        &= \sum\limits_{d \in (\N^*)^{\Sigma(1)}} \mu(d) . N_d(B) .
    \end{align*}
\end{proof}

We shall show the following theorem:

\begin{theorem}
Let $\varepsilon \in ]0,1-\frac{1}{l}[$. There exists a constant $C > 0$ such that:
\begin{align*}
    &\left| \sharp( \stackT(\Z) \cap T_{\Sigma(1)}(\R) )_{h_{\stackT,\infty} \in \mathrm{D}_B} - 2^t . \prod\limits_p \w_{\stackT,p}(\stackT(\Z_p)) . \mu_{X,\infty}(X(\R)). \nu(\mathrm{D}_1) B^{ \langle \w_{X}^{-1} , u \rangle } \right| \\
    &\leqslant C . \frac{B^{\langle \w_{X}^{-1} , u \rangle}}{B^{(1 - \frac{1}{l} - \varepsilon) \min\limits_{\rho \in \Sigma(1)} \langle [D_{\rho}] , u \rangle }}.
\end{align*}
\end{theorem}

\begin{proof}
    Recall that by propositions \ref{value_local_mesure_mobius_function} and \ref{value_global_mesure_mobius_fonction}, we have that:
    $$\sum\limits_d \frac{\mu(d)}{\Pi(d)} = \prod\limits_p \w_{\stackT,p}(\stackT(\Z_p)) .$$ Hence using the previous proposition, we get:
    \begin{align*}
        &\left| \sharp( \stackT(\Z) \cap T_{\Sigma(1)}(\R) )_{h_{\stackT,\infty} \in \mathrm{D}_B} - 2^t . \prod\limits_p \w_{\stackT,p}(\stackT(\Z_p)) . \mu_{X,\infty}(X(\R)). \nu(\mathrm{D}_1) B^{ \langle \w_{X}^{-1} , u \rangle } \right| \\
        =& \left| \sum\limits_d \mu(d) \left( N_d(B) - 2^t . \frac{1}{\Pi(d)} . \mu_{X,\infty}(X(\R)). \nu(\mathrm{D}_1) B^{ \langle \w_{X}^{-1} , u \rangle } \right) \right| \\
        \leqslant&  \sum\limits_{\Pi(d) \leqslant B^{\min\limits_{\rho \in \Sigma(1)} \langle [D_{\rho}] , u \rangle }} | \mu(d) | . \left| N_d(B) - 2^t . \frac{1}{\Pi(d)} . \mu_{X,\infty}(X(\R)). \nu(\mathrm{D}_1) B^{ \langle \w_{X}^{-1} , u \rangle } \right| \\
        +& \sum\limits_{\Pi(d) \geqslant B^{\min\limits_{\rho \in \Sigma(1)} \langle [D_{\rho}] , u \rangle }} | \mu(d) | . \left| N_d(B) - 2^t . \frac{1}{\Pi(d)} . \mu_{X,\infty}(X(\R)). \nu(\mathrm{D}_1) B^{ \langle \w_{X}^{-1} , u \rangle } \right|.
    \end{align*}
    When $\Pi(d) \leqslant B^{\min\limits_{\rho \in \Sigma(1)} \langle [D_{\rho}] , u \rangle }$ we use corollary \ref{estimation_N_d(B)}, we have that there exists a constant $C > 0$ such that:
    $$\left|  N_d(B) - \frac{1}{\Pi(d)}.2^t.\mu_{X,\infty}(X(\R)). \nu(\mathrm{D}_1) B^{ \langle \w_{X}^{-1} , u \rangle } \right| \leqslant C .\frac{B^{\langle \w_X^{-1},u\rangle}}{B^{ \min\limits_{\rho \in \Sigma(1)} \langle [D_{\rho}],u\rangle } } .$$Let $\varepsilon \in ]0,1-\frac{1}{l}[$. Then by lemma \ref{majoration_fonction_de_mobius}, we get that:
    \begin{align*}
        &\sum\limits_{\Pi(d) \leqslant B^{\min\limits_{\rho \in \Sigma(1)} \langle [D_{\rho}] , u \rangle }} | \mu(d) | . \left| N_d(B) - 2^t . \frac{1}{\Pi(d)} . \mu_{X,\infty}(X(\R)). \nu(\mathrm{D}_1) B^{ \langle \w_{X}^{-1} , u \rangle } \right| \\
         &\leqslant C.\frac{B^{\langle \w_X^{-1},u\rangle}}{B^{ \min\limits_{\rho \in \Sigma(1)} \langle [D_{\rho}],u\rangle } } . B^{(\frac{1}{l} + \varepsilon).\min\limits_{\rho \in \Sigma(1)} \langle [D_{\rho}],u\rangle }.
    \end{align*}
    Now by proposition \ref{upper_bound_N_d(B)}, there exists $C' > 0$ such that:
    $$N_d(B) \leqslant C'.\frac{B^{\langle \w_{X}^{-1} , u \rangle}}{\Pi(d)} .$$ Hence we get using the second assertion of lemma \ref{majoration_fonction_de_mobius} that there exists $C'' > 0$ such that:
    \begin{align*}
        &\sum\limits_{\Pi(d) \geqslant B^{\min\limits_{\rho \in \Sigma(1)} \langle [D_{\rho}] , u \rangle }} | \mu(d) | . \left| N_d(B) - 2^t . \frac{1}{\Pi(d)} . \mu_{X,\infty}(X(\R)). \nu(\mathrm{D}_1) B^{ \langle \w_{X}^{-1} , u \rangle } \right| \\
        &\leqslant C''. B^{\langle \w_X^{-1},u\rangle} .\sum\limits_{\Pi(d) \geqslant B^{\min\limits_{\rho \in \Sigma(1)} \langle [D_{\rho}] , u \rangle }} \frac{| \mu(d) |}{\Pi(d)} \\
        &\leqslant C''.\frac{B^{\langle \w_X^{-1},u\rangle}}{B^{(1-\frac{1}{l} - \varepsilon) \min\limits_{\rho \in \Sigma(1)} \langle [D_{\rho}],u\rangle } } .
    \end{align*}This concludes our proof.    
\end{proof}

Applying theorem \ref{number_of_points_complementary_of_the_torus_universal_torsor}, we get:

\begin{cor}\label{conclusion_torsor_univ}
Let $\varepsilon \in ]0,1-\frac{1}{l}[$. There exists a constant $C > 0$ such that:
\vspace{0,1 cm}
    $$\left| (\stackT(\Z) )_{h_{\stackT,\infty} \in \mathrm{D}_B}  - 2^t . \prod\limits_p \w_{\stackT,p}(\stackT(\Z_p)) . \mu_{X,\infty}(X(\R)). \nu(\mathrm{D}_1) B^{ \langle \w_{X}^{-1} , u \rangle } \right| \leqslant C . \frac{B^{\langle \w_{X}^{-1} , u \rangle}}{B^{(1 - \frac{1}{l} - \varepsilon) \min\limits_{\rho \in \Sigma(1)} \langle [D_{\rho}] , u \rangle }}$$
\end{cor}

\subsection{Conclusion}

Finally we can establish the theorem \ref{theorem_principal_multi_hauteur}:

\begin{theorem}\label{theorem_final_variete_torique}
Let $\epsilon \in ]0,1 - \frac{1}{l}[$. Then we have:
\vspace{0,1 cm}
$$ \sharp ( X(\Q) )_{h \in \mathrm{D}_B} \underset{B \rightarrow +\infty}{=} \nu(\mathrm{D}_1 ) . \tau(X) .  B^{\langle \w_X^{-1},u \rangle} \left( 1 + O\left(B^{- (1 - \frac{1}{l} - \varepsilon).\min\limits_{\rho \in \Sigma(1)} \langle [D_{\rho}] , u \rangle} \right) \right)$$where $\nu$ is the measure over $\Pic(X)^{\vee}_{\R}$ defined as in \ref{mesure_picard_group}.
\end{theorem}

\begin{proof}
      Let us recall that $\sharp X(\Q)_{h \in \mathrm{D}_B} = \frac{1}{2^t} \sharp \left( \stackT(\Z) \right)_{h_{\stackT,\infty} \in \mathrm{D}_B}$. Applying corollary \ref{conclusion_torsor_univ}, we get that there exists a constant $C > 0$ such that:
       $$\left| (\sharp X(\Q)_{h \in \mathrm{D}_B}  - \left( \prod\limits_p \w_{\stackT,p}(\stackT(\Z_p)) \right) . \mu_{X,\infty}(X(\R)). \nu(\mathrm{D}_1) B^{ \langle \w_{X}^{-1} , u \rangle } \right| \leqslant C . \frac{B^{\langle \w_{X}^{-1} , u \rangle}}{B^{(1 - \frac{1}{l} - \varepsilon) \min\limits_{\rho \in \Sigma(1)} \langle [D_{\rho}] , u \rangle }} .$$Moreover by remark \ref{tamagawa_number_toric_variety} and theorem \ref{mesure_torseur_univ_tamagawa_number_toric}, we get: $$  \left( \prod\limits_p \w_{\stackT,p}(\stackT(\Z_p)) \right) . \mu_{X,\infty}(X(\R)) = \tau(X)$$hence we have the result stated.
\end{proof}

%\section{Multi-height study of Campana points of toric varieties}

%\input{campana_point/main}


\bibliographystyle{amsplain}
\bibliography{bibliography}

\end{document}